\documentclass[reqno, 12pt]{amsart}
\pdfoutput=1
\makeatletter
\let\origsection=\section \def\section{\@ifstar{\origsection*}{\mysection}} 
\def\mysection{\@startsection{section}{1}\z@{.7\linespacing\@plus\linespacing}{.5\linespacing}{\normalfont\scshape\centering\S}}
\makeatother        

\usepackage{amsmath,amssymb,amsthm}
\usepackage{mathrsfs}
\usepackage{mathabx}\changenotsign
\usepackage{dsfont}
 
\usepackage{xcolor}
\usepackage[backref]{hyperref}
\hypersetup{
    colorlinks,
    linkcolor={red!60!black},
    citecolor={green!60!black},
    urlcolor={blue!60!black}
}

\usepackage{graphicx}

\usepackage[open,openlevel=2,atend]{bookmark}

\usepackage[abbrev,msc-links,backrefs]{amsrefs} 
\usepackage{doi}

\renewcommand{\PrintDOI}[1]{\doi{#1}}

\usepackage[T1]{fontenc}
\usepackage{lmodern}
\usepackage[babel]{microtype}
\usepackage[english]{babel}

\usepackage{geometry}
\numberwithin{equation}{section}
\numberwithin{figure}{section}

\usepackage{enumitem}
\def\rmlabel{\upshape({\itshape \roman*\,})}

\let\polishlcross=\l
\def\l{\ifmmode\ell\else\polishlcross\fi}

\def\qand{\quad\text{and}\quad}
\def\qqand{\qquad\text{and}\qquad}

\def\paragraph#1{\noindent\textbf{#1.}\enspace}

\let\emptyset=\varnothing
\let\setminus=\smallsetminus

\let\sm=\setminus

\makeatletter
\def\moverlay{\mathpalette\mov@rlay}
\def\mov@rlay#1#2{\leavevmode\vtop{   \baselineskip\z@skip \lineskiplimit-\maxdimen
   \ialign{\hfil$\m@th#1##$\hfil\cr#2\crcr}}}
\newcommand{\charfusion}[3][\mathord]{
    #1{\ifx#1\mathop\vphantom{#2}\fi
        \mathpalette\mov@rlay{#2\cr#3}
      }
    \ifx#1\mathop\expandafter\displaylimits\fi}
\makeatother

\newcommand{\dcup}{\charfusion[\mathbin]{\cup}{\cdot}}

\DeclareFontFamily{U}  {MnSymbolC}{}
\DeclareSymbolFont{MnSyC}         {U}  {MnSymbolC}{m}{n}
\DeclareFontShape{U}{MnSymbolC}{m}{n}{
    <-6>  MnSymbolC5
   <6-7>  MnSymbolC6
   <7-8>  MnSymbolC7
   <8-9>  MnSymbolC8
   <9-10> MnSymbolC9
  <10-12> MnSymbolC10
  <12->   MnSymbolC12}{}
\DeclareMathSymbol{\powerset}{\mathord}{MnSyC}{180}

\usepackage{tikz}
\usetikzlibrary{calc,decorations.pathmorphing}
\pgfdeclarelayer{background}
\pgfdeclarelayer{foreground}
\pgfdeclarelayer{front}
\pgfsetlayers{background,main,foreground,front}

\let\epsilon=\varepsilon
\let\eps=\epsilon
\let\rho=\varrho
\let\theta=\vartheta
\let\kappa=\varkappa

\def\EE{{\mathds E}}

\def\NN{{\mathds N}}

\def\PP{{\mathds P}}

\def\RR{{\mathds R}}

\newcommand{\cC}{\mathcal{C}}
\newcommand{\cE}{\mathcal{E}}
\newcommand{\cF}{\mathcal{F}}

\newcommand{\cK}{\mathcal{K}}
\newcommand{\cL}{\mathcal{L}}

\newcommand{\cP}{\mathcal{P}}

\newcommand{\ccJ}{\mathscr{J}}

\def\hp{\hat p}
\def\eu{\mathrm{e}}
\def\lto{\longrightarrow}

\theoremstyle{plain}
\newtheorem{thm}{Theorem}[section]
\newtheorem{theorem}[thm]{Theorem}
\newtheorem{corollary}[thm]{Corollary}

\newtheorem{prop}[thm]{Proposition}
\newtheorem{claim}[thm]{Claim}

\newtheorem{lemma}[thm]{Lemma}

\theoremstyle{definition}

\usepackage{accents}
\newcommand{\seq}[1]{\accentset{\rightharpoonup}{#1}}

\let\phi=\varphi

\begin{document}

\title[Powers of Hamiltonian cycles in randomly augmented graphs]{Powers of Hamiltonian cycles in randomly \\ augmented graphs}

\author{Andrzej Dudek}
\address{Department of Mathematics, Western Michigan University, Kalamazoo, MI, USA}
\email{andrzej.dudek@wmich.edu}
\thanks{The first author was supported in part by Simons Foundation Grant \#522400.}

\author{Christian Reiher}
\address{Fachbereich Mathematik, Universit\"at Hamburg, Hamburg, Germany}
\email{christian.reiher@uni-hamburg.de}
\thanks{The second and fourth author were supported by the 
European Research Council (PEPCo 724903).}

\author{Andrzej Ruci\'nski}
\address{Department of Discrete Mathematics, Adam Mickiewicz University, Pozna\'n, Poland}
\email{rucinski@amu.edu.pl}
\thanks{The third author was supported in part by the Polish NSC grant 2014/15/B/ST1/01688}

\author{Mathias Schacht}
\address{Fachbereich Mathematik, Universit\"at Hamburg, Hamburg, Germany}
\email{schacht@math.uni-hamburg.de}

\begin{abstract} 
We study the existence of powers of Hamiltonian cycles in graphs with large minimum degree 
to which some additional edges have been added in a random manner. It follows from the theorems of 
Dirac and of  Koml\'os, Sark\"ozy, and Szemer\'edi that for every $k\geq 1$ and sufficiently large $n$
already the minimum degree $\delta(G)\ge\tfrac{k}{k+1}n$ for an $n$-vertex graph $G$ alone suffices to 
ensure the existence of a $k$-th power of a Hamiltonian cycle.
Here we show that under essentially the same degree assumption the addition of just $O(n)$ 
random edges ensures the presence of the $(k+1)$-st power of a Hamiltonian cycle
with probability close to one.
\end{abstract}

\maketitle

\setcounter{footnote}{1}

\section{Introduction}
All graphs we consider are finite and for simplicity we assume that the vertex set $V$   of any given graph is 
the set $\{1,\dots,|V|\}$. We recall that for $k\in\NN$ the \emph{$k$-th power~$H^k$} of a~graph~$H$ is defined  
to be a graph on the same vertex set, where edges in $H^k$ signify that its vertices have distance at most~$k$ in~$H$. 
Consequently, $H^0$ is the empty graph on the same vertex set and $H^1=H$. 

For integers $n\geq k+2$ and $k\geq 1$ we consider the set of graphs $\cP_n^k$ consisting 
of all $n$-vertex graphs~$G$ that contain the $k$-th power of a Hamiltonian cycle and we 
set $\cP^k=\bigcup_{n\geq k+2} \cP^k_n$. Clearly, 
$\cP_n^k$ is a monotone graph property for fixed $n$ and $k$, as powers of a Hamiltonian cycle cannot disappear 
by adding edges to a graph without adding new vertices.  

We investigate the probabilities that a given $n$-vertex graph $G$ with high minimum degree
augmented by a binomial random graph $G(n,p)$ spans a 
$k$-th power of a Hamiltonian cycle, i.e., we are interested in~$\PP(G\cup G(n,p)\in\cP_n^k)$.
More formally, for $\alpha\in [0,1)$ and~$p\colon \NN\to[0,1]$ we say $(\alpha,p)$ \emph{ensures} $\cP^k$ if 
\[
	\lim_{n\to\infty}\min_G\PP\big(G\cup G(n,p(n))\in\cP_n^k\big)=1\,,
\] 
where the minimum is taken over all $n$-vertex graphs~$G$ with $\delta(G)\geq \alpha n$.
We are interested in the `minimal' pairs $(\alpha,p)$ that  ensure $\cP^k$. 

For example, when $p=0$, then this reduces to the classical theorem of Dirac~\cite{D1952} on Hamiltonian cycles 
for $k=1$ and for $k\geq 2$ to the P\'osa--Seymour conjecture~\cites{E1964,S1974} and its resolution (for large $n$) by 
Koml\'os, Sark\"ozy, and Szemer\'edi~\cite{KSS1998}. These beautiful results then assert that 
$(\frac{k}{k+1},0)$ ensures $\cP^k$ for every $k\geq 1$. 

For the other extreme case, when $\alpha=0$, we arrive at the threshold problem for the existence of powers 
of Hamiltonian cycles in $G(n,p)$. This was asymptotically solved by Pos\'a~\cite{P1976} for $k=1$ 
(see also~\cites{AKS1985, B1984} for sharper results). For $k=2$ the threshold is only known up to a factor of ${\rm poly}(\log n)$ due 
to Nenadov and \v{S}kori\'c~\cite{NS2018}. For $k\geq 3$ the threshold is given by a result of Riordan~\cite{R2000}, which was observed by K\"uhn and Osthus~\cite{KO2012}.
Writing~$\hp_k(n)$ for the threshold for $\cP^k$ then these results can be summarised by
\[
\hp_1(n)\sim\frac{\ln n}{n}\,,\qquad
\left(\frac{\eu}{n}\right)^{\frac{1}{2}}\leq\hp_2(n)=O\Big(\frac{(\ln n)^4}{\sqrt{n}}\Big)\,,
\qqand
\hp_k(n)\sim\left(\frac{\eu}{n}\right)^{\frac{1}{k}}\ \text{for}\ k\geq 3\,,
\] 
where $\ln$ stands for the natural logarithm $\log_\eu$.

We study for $\alpha>\frac{k}{k+1}$ the asymptotics of the smallest function $p=p(n)$ 
such that $(\alpha,p)$ ensures $\cP^{k+1}$.
Recall that for $\alpha> \frac{k}{k+1}$ the Koml\'os--Sark\"ozy--Szemer\'edi theorem asserts that 
$(\alpha,0)$ ensures $\cP^k$ already. We show that under the same minimum degree assumption  for $n$-vertex graphs~$G$
the addition of $O(n)$ random edges suffices to ensure~$\cP^{k+1}$, which is asymptotically best possible 
(see discussion below).

\begin{theorem}\label{thm:main}
For every integer $k\in\NN$ and  every $\alpha\in\RR$ with $\tfrac k{k+1}<\alpha<1$ there is some  constant $C=C(k,\alpha)$ such that for $p=p(n)\geq C/n$
the pair $(\alpha, p)$ ensures $\cP^{k+1}$.
\end{theorem}

For $k=0$ Theorem~\ref{thm:main} was already obtained by Bohman, Frieze, and Martin~\cite{BFM2003}. For 
larger~$k$ only suboptimal upper bounds for $p(n)$ were established so far. The best known bound of 
the form $\hp_{k+1}(n)/n^{\delta}$ for some $\delta>0$ and $k\geq2$
was given by Bedenknecht, Han, Kohayakawa, and Mota~\cite{BHKM2018} for $k\geq 2$ (see also~\cite{BMPP}). 

The following construction shows that Theorem~\ref{thm:main} is optimal in the sense that for every 
$\alpha>\frac{k}{k+1}$ there are $n$-vertex graphs $G$ with $\delta(G)/n\geq \alpha>\frac{k}{k+1}$ that require 
at least $\Omega(n)$ additional random edges to ensure a $(k+1)$-st power 
of a Hamiltonian cycle. 

Let $(k+1)\mid n$ and consider a vertex partition $[n]=V_1\dcup\cdots\dcup V_{k+1}$ 
with each part of size~$n/(k+1)$. Moreover, for every $i=1,\dots,k+1$ fix some subset $W_i\subseteq V_i$ of 
size~$|W_i|=\lceil \eps n\rceil$  
for some arbitrarily small $\eps>0$.  

Let $G$ be the $n$-vertex graph consisting of the union of the complete $(k+1)$-partite graph 
with vertex partition~$V_1\dcup \dots\dcup V_{k+1}$ and $k+1$ complete bipartite graphs with vertex classes~$W_i$ and~$V_i\setminus W_i$ for $i=1,\dots,k+1$. 
Clearly, $\delta(G)\geq (\tfrac{k}{k+1}+\eps)n$. However, any copy of~$\cC^{k+1}_n$, the $(k+1)$-st power of a Hamiltonian cycle, 
contains $\lfloor n/(k+2)\rfloor$ vertex-disjoint copies of~$K_{k+2}$ and each of these cliques would require at least 
one edge contained in some set~$V_i$. Consequently, every such clique has at least one vertex in 
$\bigcup_{i=1}^{k+1} W_i$ and, hence,~$G$ contains at most 
$\big|\bigcup_{i=1}^{k+1} W_i\big|=(k+1)\lceil \eps n\rceil$ vertex disjoint $K_{k+2}$'s. This implies that for~$\eps\ll (k+1)^{-2}$ one needs to add at least a matching of size $\Omega(n)$ to~$G$ before it may have a chance to contain a copy of $\cC^{k+1}_n$.

In view of the optimality of Theorem~\ref{thm:main}, the next open problem might be to find the asymptotics of the minimal $p$ such that $(\alpha,p)$ ensures $\cP^{k+1}$ for $\alpha=\frac{k}{k+1}$ or even smaller values of $\alpha$. This problem was also considered by B\"ottcher, Montgomery, Parczyk, and Person in~\cite{BMPP}.

\section{Method of Absorption}
\label{sec:absM}

The proof of Theorem~\ref{thm:main} is based on the {\it absorption method}, which has been 
introduced about a decade ago  in~\cite{rrs3}. Since then, it has turned out to be an 
extremely versatile technique for solving a variety of combinatorial problems 
concerning the existence of spanning substructures in graphs and hypergraphs obeying 
minimum degree conditions. 

A nice feature of this method is that it often makes it possible to split the problem at hand into
several subproblems, which may turn out to be more manageable. In the present case, we may 
reduce Theorem~\ref{thm:main} to the 
Propositions~\ref{prop:connect}--\ref{prop:covering} formulated later in this section.

Before stating the first of these propositions, we fix some terminology concerning powers
of paths. A \emph{$(k+1)$-path} is defined as the $(k+1)$-st power of a path. The ordered 
sets of the first and last $k+1$ vertices are called the \emph{end-sets} of the $(k+1)$-path, 
which must span $(k+1)$-cliques. If $K$ and $K'$ are the ordered cliques induced by the 
end-sets of a $(k+1)$-path~$P$, we say that~$P$ \emph{connects} $K$ and $K'$ and the vertices 
of $P$ not contained in~$V(K)\cup V(K')$ are its \emph{internal} vertices. 

We may now state the so-called Connecting Lemma, which is proved in Section~\ref{sec:connect}. 
Roughly speaking, it asserts that in the graphs we need to deal with, one may connect any two 
disjoint $(k+1)$-cliques by means of a ``short'' $(k+1)$-path. 
Moreover, we want to declare some small proportion of the vertex set to 
be ``unavailable'' for such a connection (e.g.\ because we already have something else in 
mind that we want to do with those vertices), then the desired connection does still exist.

\begin{prop}[Connecting Lemma]\label{prop:connect}
For every integer $k\geq 0$ and every $\eps>0$ there exists 
some $C>1$ such that for 
every $n$-vertex graph $G$ with $\delta(G)\geq (\tfrac{k}{k+1}+\eps)n$ and $p=p(n)\geq C/n$ 
a.a.s.\ $H=G\cup G(n,p)$ has the following property:\footnote{As usual \emph{a.a.s.} abbreviates \emph{asymptotically almost surely} and 
means that the statement holds with \emph{probability tending to~$1$ as $n\to\infty$}. 
Strictly speaking, we should therefore consider arbitrary sequences~$(G_n)_{n\in\NN}$ 
of $n$-vertex graphs with $\delta(G_n)\geq (\tfrac{k}{k+1}+\eps)n$. 
However, for a less baroque presentation we chose this `simplification' here and in the propositions below.}

\nopagebreak
For every subset $Z\subseteq V$ of size at most $\eps n/2$ and every pair of 
disjoint, ordered $(k+1)$-cliques $K$, $K'$, there exists a $(k+1)$-path connecting $K$ and $K'$ with
exactly $(k+1)2^{k+1}$ internal vertices from $V\setminus Z$.
\end{prop}

As the proof of Theorem~\ref{thm:main} progresses, the number of vertices we do not want to
use for connections anymore gets out of control. Therefore one puts a small set $R$ of 
vertices aside at the beginning, which is called the {\it reservoir} and has the property 
that, actually, we can always connect any two given $(k+1)$-cliques through the reservoir. 
Of course, in order to use the reservoir multiple times, we shall need again a version, where 
a small part of the reservoir is ``unavailable'' at any particular moment. 
A precise version of the Reservoir Lemma, which is proved in Section~\ref{sec:reservoir},
reads as follows.   

\begin{prop}[Reservoir Lemma]\label{prop:reservoir}
For every integer $k\geq 0$  and every $\eps>0$, $\gamma\in(0,1)$ 
there exists $C>1$ such that for every $n$-vertex graph $G$ with 
$\delta(G)\geq (\tfrac{k}{k+1}+\eps)n$ there exists a set of vertices $R\subseteq V$ 
of size $\gamma^2n$
such that for $p=p(n)\geq C/n$ a.a.s.\ $H=G\cup G(n,p)$ has the following property:

For every $S\subseteq R$ with  $|S|\leq \eps|R|/4$ and for every pair of disjoint, 
ordered $(k+1)$-cliques $K$, $K'$ in~$G-R$, there exists a~$(k+1)$-path connecting $K$ and $K'$ 
with exactly~$(k+1)2^{k+1}$ internal vertices from $R\setminus S$.   
\end{prop}

The next result (proved in Section~\ref{sec:absorb}) plays a central r\^ole and, in fact, 
this kind of statement gave the absorption method its name. It promises the existence of a 
very special, so-called {\it absorbing $(k+1)$-path} $A$, which can `absorb' any small set of 
vertices. Thus the problem of constructing the $(k+1)$-st power of a Hamiltonian cycle gets 
reduced to the much easier problem of finding the $(k+1)$-st power of an almost spanning cycle
containing~$A$. Let us remark at this point that the Absorbing Lemma gets utilised after the 
Reservoir Lemma and the set $R$ appearing below takes this fact into account.  

\begin{prop}[Absorbing Lemma]\label{prop:absorbing}
For every integer $k\ge 0$ and every $\eps>0$ there exist~$\gamma\in(0, \eps/4^{k+2})$ 
and $C>1$ such that 
for every $n$-vertex graph $G$ with $\delta(G)\geq (\tfrac{k}{k+1}+\eps)n$
and every $p=p(n)\geq C/n$ a.a.s.\ $H=G\cup G(n,p)$ has the following property:

\nopagebreak
For every set of vertices $R\subseteq V$ of size $\gamma^2n$ the graph $H-R$ contains
a $(k+1)$-path $A$ with at most $\gamma n/2$ vertices such that for every $U\subseteq V$
with $|U|\le 2\gamma^2n$ the graph $H[V(A)\cup U]$ contains a spanning $(k+1)$-path having 
the same end-sets as $A$.
\end{prop}

The last ingredient of our argument is a statement to the effect that essentially 
the whole graph under consideration can be covered by ``not too many'' $(k+1)$-paths. 
Such paths can be connected together with the absorbing path $A$ obtained earlier 
by means of ``relatively few'' connections to be made through the reservoir, thus 
producing the desired $(k+1)$-st power of an almost spanning cycle. We shall prove this 
Covering Lemma in Section~\ref{sec:cover}.

\begin{prop}[Covering Lemma]\label{prop:covering}
For every integer $k\geq 0$ and every $\eps>0$, $\gamma\in(0,\eps/2]$ there exists $C>1$
such that for every $n$-vertex graph $G$ with $\delta(G)\geq (\tfrac{k}{k+1}+\eps)n$ and  
$p=p(n)\geq C/n$ a.a.s.\ $H=G\cup G(n,p)$ has the following property:

\nopagebreak
For every subset $Q\subseteq V$ of size at most $\gamma n$ there exists a family 
of $\gamma^{3} n$ 
vertex disjoint $(k+1)$-paths in $H-Q$ that cover all but at most $\gamma^2n$ vertices from 
$V\setminus Q$.
\end{prop}

We conclude the present section with a proof of our main result assuming the four 
propositions stated above. In fact, we shall not make a direct reference to 
Proposition~\ref{prop:connect} in the proof below, but it will be employed in the proof
of Proposition~\ref{prop:absorbing} in Section~\ref{sec:absorb}.
 
\begin{proof}[Proof of Theorem~\ref{thm:main}]
	Let $k\in \NN$ and $\alpha\in\bigl(\frac k{k+1}, 1\bigr)$ be given and 
	set~$\eps=\alpha-\frac k{k+1}$. Plugging~$k$ and $\eps$ into 
	Proposition~\ref{prop:absorbing} we get $\gamma\in(0, \eps/4^{k+2})$ and $C_3>1$.
	Next we appeal with $k$, $\eps$, and $\gamma$ to 
	Propositions~\ref{prop:reservoir} and~\ref{prop:covering}, thus getting 
	two further constants $C_2>1$ and $C_4>1$.
	We claim that $C=\max\{C_2, C_3, C_4\}$ is as desired. 
	
	So let an $n$-vertex graph $G$ with $\delta(G)\geq (\tfrac{k}{k+1}+\eps)n$
	as well as some $p\ge C/n$ be given. We need to check that a.a.s.\ the graph 
	$H=G\cup G(n,p)$ contains the $(k+1)$-st power of a Hamiltonian cycle $\cC_n^{k+1}$. 
	For this purpose it suffices to prove that every graph $H=G\cup G(n,p)$
	exemplifying the conclusion of Proposition~$\color{red!60!black}{2.x}$ for each 
	$x\in\{2, 3, 4\}$ contains a copy of~$\cC_n^{k+1}$.

	Use Proposition~\ref{prop:reservoir} for obtaining a reservoir set $R\subseteq V$
	of size $\gamma^2n$. By Proposition~\ref{prop:absorbing} there exists an absorbing 
	$(k+1)$-path $A\subseteq H-R$. Since $|R|+|V(A)|\le(\gamma^2+\gamma/2)n<\gamma n$,
	we can apply Proposition~\ref{prop:covering} to $Q=R\cup V(A)$ and obtain a 
	collection $\cP$ of at most $\gamma^3n$ vertex-disjoint $(k+1)$-paths covering 
	the whole graph $H-Q$ except for a small set of vertices $U_\star\subseteq V\sm Q$
	with $|U_\star|\le \gamma^2n$. Now we want to create the $(k+1)$-st power of a cycle 
	$\cC\subseteq H$ 
	\begin{enumerate}
		\item[$\bullet$] containing $A$ and each $(k+1)$-path in $\cP$ as a sub-path,
		\item[$\bullet$] such that between any two ``consecutive'' such sub-paths of 
			$\cC$ there are always exactly $(k+1)2^{k+1}$ vertices from $R$. 
	\end{enumerate} 
	
	For building $\cC$ we intend to make $|\cP|+1$ successive applications of 
	Proposition~\ref{prop:reservoir}. In each such application we let $K$ and $K'$
	be the end-sets of $(k+1)$-paths we wish to connect and we let $S\subseteq R$
	be the set of all vertices that we obtained as internal vertices in previous 
	applications of Proposition~\ref{prop:reservoir}. When arriving at the last step of this 
	process closing the cycle~$\cC$, the set $S$ of vertices we need to exclude has size 
\[
		|S|=(k+1)2^{k+1}\cdot |\cP|\le 4^{k+1}\gamma^3 n\le \frac{\eps}{4}|R|\,,
	\]
which justifies the applications of Proposition~\ref{prop:reservoir}. 
	
	Now the complement $U=V\sm V(\cC)$ satisfies $|U|=|U_\star|+|R\sm V(\cC)|\le 2\gamma^2n$, 
	whence by Proposition~\ref{prop:absorbing} there exists a $(k+1)$-path $A_U$
	with $V(A_U)=V(A)\dcup U$ having the same end-sets as $A$. Therefore, we can replace $A$
	by $A_U$ in $\cC$ and obtain the desired $(k+1)$-st power of a Hamiltonian cycle  
	$\cC_n^{k+1}\subseteq H$. 
\end{proof}

\section{Preliminaries} In the proofs of the
propositions stated in Section~\ref{sec:absM}
we make use of the high minimum degree condition of the given graph~$G$ and combine it with 
properties of $G(n,p)$. We prepare for this by collecting a few observations for such graphs~$G$
in Section~\ref{sec:degrees} and for the random graph in Section~\ref{sec:gnp} below.

\subsection{Neighbourhoods in graphs of large minimum degree}\label{sec:degrees}
We recall the following standard notation. For a set $V$ and an integer $j\in\NN$ we write $V^{(j)}$ for the set of all $j$-element subsets of $V$. Given a graph $G=(V,E)$ we write $N_G(u)$ 
for the neighbourhood of a vertex~$u\in V$.
More generally, for a subset $U\subseteq V$ we set 
\[
	N_G(U)=\bigcap_{u\in U}N(u)
\] 
for the \emph{joint neighbourhood} of $U$. For simplicity we may suppress $G$ in the subscript and 
for sets $\{u_1,\dots,u_r\}$ we may write $N(u_1,\dots,u_r)$ instead of $N(\{u_1,\dots,u_r\})$.

\begin{lemma}\label{lem:scale}
	For every integer $k\geq 0$ and $\eps>0$ the following holds for every $n$-vertex graph $G=(V,E)$
	with $\delta(G)\geq(\tfrac{k}{k+1}+\eps)n$. For every $j\in[k+1]$ and every $J\in V^{(j)}$ 
 	we have
\begin{equation}\label{claim1:N}
		|N(J)|\ge\left(\frac{k+1-j}{k+1}+j\eps \right)n\,.
	\end{equation}
Furthermore, for $j\in[k]$ the induced subgraph $G[N(J)]$ satisfies
\begin{equation}\label{claim1:d}
		\delta(G[N(J)])
		\ge
		\left(\frac {k-j}{k-j+1}+\eps \right)|N(J)|
	\end{equation}
	for every $J\in V^{(j)}$.
	\end{lemma}

\begin{proof}
	First observe that De Morgan's law and Boole's inequality imply 
	\[
		n-\big|N(J)\big|
		=
		\big|V\setminus N(J)\big|
		=
		\bigg|\bigcup_{u\in J}\big(V\sm N(u)\big)\bigg|
		\leq
		jn-\sum_{u\in J}\big|N(u)\big|\,.
	\]
Therefore, 
\begin{multline*}
		|N(J)| 
		\ge 
		\sum_{u\in J}\big|N(u)\big| - (j-1)n 
		\ge 
		j \delta(G) - (j-1)n\\
		\ge
		\Big(\frac{jk}{k+1}+j\eps\Big)n-(j-1)n
		=
		\Big(\frac{jk-(j-1)(k+1)}{k+1}+j\eps\Big)n
		\,,
	\end{multline*}
which yields~\eqref{claim1:N}.

	Proceeding with~\eqref{claim1:d} we note that every $v\in N(J)$ satisfies
\[
		\big|N(v)\cap N(J)\big|
		\ge
		\delta(G)-\big(n-|N(J)|\big)
		\geq
		\left(1 - \frac{n-\delta(G)}{|N(J)|}\right)|N(J)|\,.
	\]
Owing to the lower bound on $\delta(G)$
	and that \eqref{claim1:N}  implies  
	$|N(J)|\geq \frac{k+1-j}{k+1}n$
	we deduce
\[
		\frac{\delta(G[N(J)])}{|N(J)|}
		\geq
		1 - \frac{n-\delta(G)}{|N(J)|}
		\geq 
		1-\frac{1-\frac{k}{k+1}-\eps}{\frac{k+1-j}{k+1}}
		\geq 
		1-\frac{1}{k+1-j}+\eps
		=
		\frac{k-j}{k+1-j}+\eps\,,
	\]
	as desired.
\end{proof}

\subsection{Janson's inequalities}\label{sec:gnp}
We shall use the following variant of Janson's inequality~\cite{J90} (see also~\cite{JLR-ineq}). 
\begin{theorem}[Janson's inequality] 
	\label{thm:janson}
Let $\rho>0$ and $C>1$ be constants. Let $F=(V_F,E_F)$ be a forest 
and let~$\cF$ be a family of copies of $F$ in $K_n$ with 
$|\cF|\ge\rho n^{|V_F|}$. 

There exists some constant $c_F$ only depending on $F$ such that for $p\geq C/n$ the 
probability that $G(n,p)$ contains no copy of 
$F$ from $\cF$ is at most $2^{-c_F\rho^2pn^2}$.\qed
\end{theorem}

The following further customised version of Janson's inequality will be utilised in our proof in Sections~\ref{sec:connect} and~\ref{sec:reservoir}. Roughly speaking this version will guarantee
that $G(n,p)$ provides the missing edges of a $(k+1)$-path connecting two 
$(k+1)$-cliques $K$ and $K'$ provided the deterministic graph $G$ guarantees many short 
$k$-paths between $K$ and $K'$. 

\begin{corollary}\label{cor:janson2}
	For all integers  $k$, $\l\geq 0$ with $(k+1)\mid\l$ and $\rho>0$ there exists $C>0$ 
	such that for every $n$-vertex graph $G$ and $p\geq C/n$ the graph  
	$H=G\cup G(n,p)$ satisfies with probability at least $1-4^{-n}$  the following property:
	
	If for a pair of ordered, disjoint $(k+1)$-cliques $K$, $K'$~in $G$ there is a family $\cP$
	of at least $\rho n^{\l+2k+2}$ $k$-paths $P=x_1\dots x_{k+1}y_1\dots y_\l x'_1\dots x'_{k+1}$ 
	in $G$ such that $Kx_1\dots x_{k+1}$ and $x'_1\dots x'_{k+1}K'$ 
	form $(k+1)$-paths, then there is at least one $k$-path $P\in\cP$ such that 
	$KPK'$ forms a $(k+1)$-path in~$H$.
\end{corollary}
\begin{proof}
	Let $F$ denote the linear forest on $\l+2k+2$ vertices consisting of $k+1$
	disjoint paths on $2+\l/(k+1)$ vertices each. For each $P\in \cP$ there is a copy 
	$F_P$ of $F$ such that the union $P\cup F_P$ forms a $(k+1)$-path connecting 
	$K$ with $K'$ (see Figure~\ref{fig:completion}). 
	
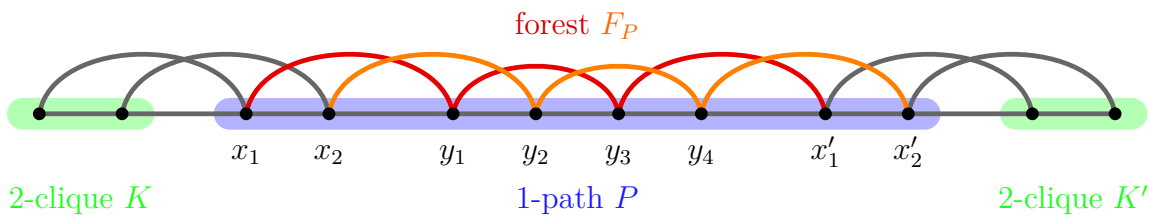
\begin{figure}[ht]
\centering
\begin{tikzpicture}[scale=1.1]
	\coordinate (a1) at (0,0);
	\coordinate (a2) at (1,0);
	\coordinate (x_1) at (2.5,0);
	\coordinate (x_2) at (3.5,0);
	\coordinate (y_1) at (5,0);
	\coordinate (y_2) at (6,0);
	\coordinate (y_3) at (7,0);
	\coordinate (y_4) at (8,0);
	\coordinate (x'_1) at (9.5,0);
	\coordinate (x'_2) at (10.5,0);
	\coordinate (b1) at (12,0);
	\coordinate (b2) at (13,0);
	
	\begin{pgfonlayer}{front}
		\foreach \i in {a1, a2, x_1, x_2, y_1, y_2, y_3, y_4, x'_1, x'_2, b1, b2}
			 \fill  (\i) circle (2.2pt);
		\foreach \i in {x_1, x_2, y_1, y_2, y_3, y_4, x'_1, x'_2}	 
			\node[below, style={font=\vphantom{\Large $K'_2$}}] at (\i) {$\i$};  
		\node[green!90!white, style={font=\vphantom{\Large $K'_2$}}] at (0.5,-1) {$2$-clique $K$};
		\node[green!90!white, style={font=\vphantom{\Large $K'_2$}}] at (12.5,-1) {$2$-clique $K'$};
		\node[blue!90!white, style={font=\vphantom{\Large $K'_2$}}] at (6.5,-1.0) {$1$-path $P$};
		\node[style={font=\vphantom{\Large $K'_2$}}] at (6.5,1.1) {\textcolor{red!80!black}{forest} 
			\textcolor{red!60!yellow}{$F_P$}};
	\end{pgfonlayer}
	
	\draw [line width=12pt, blue!30!white, rounded corners=5pt, line cap=round]
		($(x_1)+(180:0.2)$) -- ($(x'_2)+(0:0.2)$);
		
	\draw [line width=12pt, green!30!white, rounded corners=5pt, line cap=round]
		($(a1)+(180:0.2)$) -- ($(a2)+(0:0.2)$);
	\draw [line width=12pt, green!30!white, rounded corners=5pt, line cap=round]
		($(b1)+(180:0.2)$) -- ($(b2)+(0:0.2)$);
	
	\draw[black!60!white, line width=1.8pt] (a1) -- (b2);
	\draw[black!60!white, line width=1.8pt] (a1) edge[out=80,in=100] (x_1);
	\draw[black!60!white, line width=1.8pt] (a2) edge[out=80,in=100] (x_2);
	\draw[black!60!white, line width=1.8pt] (x'_1) edge[out=80,in=100] (b1);
	\draw[black!60!white, line width=1.8pt] (x'_2) edge[out=80,in=100] (b2);
	
	\draw[red!90!black, line width=1.8pt] (x_1) edge[out=80,in=100] (y_1);
	\draw[red!90!black, line width=1.8pt] (y_1)	edge[out=80,in=100] (y_3);
	\draw[red!90!black, line width=1.8pt] (y_3)	edge[out=80,in=100] (x'_1);
	
	\draw[red!50!yellow, line width=1.8pt] (x_2) edge[out=80,in=100] (y_2);
	\draw[red!50!yellow, line width=1.8pt] (y_2)	edge[out=80,in=100] (y_4);
	\draw[red!50!yellow, line width=1.8pt] (y_4)	edge[out=80,in=100] (x'_2);
	
\end{tikzpicture}
\caption{For $k=1$ and $\l=4$ completing a $1$-path $P$ to a $2$-path 
	with a linear forest $F_P$ consisting of two $3$-edge paths.}
\label{fig:completion}
\end{figure}
	
	We estimate the probability that at least one of them is a subgraph of $G(n, p)$.
	Setting $\cF=\{F_P\colon P\in \cP\}$ we 
	have~$|\cF|=|\cP|\ge \rho n^{\l+2k+2}$, Theorem~\ref{thm:janson}
	shows that for $C\geq 2c^{-1}_F\rho^{-2}$ this leads to 
\[
		\PP\big(F_P\not\subseteq G(n, p) \text{ for all } F_P\in \cF\big)
		\le 
		4^{-n}\,,
	\]
and the corollary is proved.
\end{proof}

\section{Proof of the Connecting Lemma} \label{sec:connect}
In this section we establish Proposition~\ref{prop:connect}.
For that we first prove a deterministic lemma (see Lemma~\ref{lem:dcon}), which guarantees 
many short
$k$-paths between every pair of disjoint $k$-cliques in large graphs $G$ with 
sufficiently high minimum degree. Similar results appeared before in~\cites{HJS,KSS98b}.
We shall employ this result in the proof of
Proposition~\ref{prop:connect}, where at least one of these $k$-paths will be `thickened' to a
$(k+1)$-path by an application of Janson's inequality in the form of Corollary~\ref{cor:janson2}. 

In Lemma~\ref{lem:dcon} below 
it will be convenient to consider \emph{$k$-walks}, which are defined like $k$-path, without the restriction 
that all vertices must be distinct. However, since we consider only graphs without loops, any $k$ consecutive vertices in a $k$-walk must be distinct. As in the case of $k$-paths we say a walk connects the ordered 
$k$-cliques forming the ends of the walk and internal vertices are counted with their
multiplicities (outside the ends).
\begin{lemma}\label{lem:dcon} 
	For every integer $k\ge1$ and $\eps>0$ there exists some $\rho_k>0$ such that 
	every $n$-vertex graph $G$ with $\delta(G)\geq (\tfrac{k}{k+1}+\eps)n$
	satisfies the following for $\l_k=(k+1)(2^{k+1}-2)$.
	
	For all pairs of disjoint, ordered $k$-cliques $K$, $K'$ in $G$ 
	the number of $k$-walks connecting~$K$ and~$K'$ with $\ell_k$ internal vertices 
	is at least $\rho_k n^{\ell_k}$.
\end{lemma}

\begin{proof}
	We argue by induction on $k$. For $k=1$ we have $\ell_1=4$ and the statement reduces to 
	showing that any two 
	distinct vertices $x$ and $y$ of an $n$-vertex graph $G$ with minimum degree 
	$\delta(G)\ge\bigl(\frac12+\eps\bigr)n$ are connected by $\rho_1 n^4$
	walks with four internal vertices for some $\rho_1=\rho_1(\eps)>0$. The minimum degree condition 
	implies that there are at least $(1/2+\eps)^3n^3$ walks with three edges that start in~$x$.
	Moreover, by~\eqref{claim1:N} for $j=2$ the end-vertex of each such walk has at least $2\eps n$ joint neighbours with~$y$, which gives 
	rise to at least $2\eps(1/2+\eps)^3n^4$ different $x$-$y$-walks in~$G$ with four internal vertices. 
	This establishes the induction start for~$\rho_1=\eps/4$.

	For the inductive step we assume that the lemma holds for $k-1$ in place of $k\ge 2$ and we consider  
	a given $n$-vertex graph $G=(V,E)$ with $\delta(G)\geq (\tfrac{k}{k+1}+\eps)n$. Given $\eps>0$ we will 
	use some auxiliary constants $\xi$, $\xi'$, $\xi''$, and $\xi'''$ before we define $\rho_k$. 
	Moreover, given $\rho_{k-1}$ by the 
	inductive assumption applied with $\eps$, we shall work under the following hierarchy of constants 
	\[
		k^{-1},\eps\gg \rho_{k-1},\xi\gg \xi'\gg \xi'' \gg \xi''' \gg \rho_k\,.
	\]
	
	First we observe that for any $u$, $w\in V(G)$ the case $j=2$ of~\eqref{claim1:N} and~\eqref{claim1:d} yields
\[
		|N(u, w)|\ge\frac{k-1}{k+1}n
		\qand
		e(N(u,w)) 
		\ge 
		\left( \frac{k-2}{k-1} + \eps \right) \frac{|N(u,w)|^2}{2}\,. 
	\]
Hence, it follows from Tur\'an's theorem that $G[N(u, w)]$ induces a copy of $K_k$ and owing 
	to the so-called \emph{supersaturation phenomenon} (see, e.g.,~\cite{ES1983})
	the induced subgraph $G[N(u, w)]$ contains $\Omega(|N(u, w)|^k)=\Omega(n^k)$ 
	copies of~$K_k$. Consequently, there exists $\xi=\xi(k,\eps)>0$ such that 
	\begin{equation}\label{eq:jcliques}
		\big|\big\{K_k\subseteq G[N(u,w)]\big\}\big|\geq \xi n^k\,,
	\end{equation}
	i.e., there are at least $\xi n^k$ copies of $K_k$ contained in~$G[N(u,w)]$ for 
	any vertices $u$, $w\in V$.

	We consider two disjoint, ordered 
	$k$-cliques $K$ and $K'$. As a preliminary step we first extend $K'$  in a greedy manner by $k$ vertices. 
	(This seems like an unnecessary step but it is needed to fulfil a certain divisibility 
	condition at the end of this proof.) The total number of these extensions is, by $k$ applications 
	of~\eqref{claim1:N} with $j=k$, at least
	\begin{equation}\label{eq:cL'}
		\bigg(\Big( \frac{1}{k+1} + k\eps \Big)n\bigg)^k \ge \Big(\frac{n}{k+1}\Big)^k\,,
	\end{equation}
	as we do not require that all these vertices are distinct from those in $K$ or $K'$. Let $\cL'$ be the set of 
	all ordered $k$-tuples obtained this way. By construction for every $L'\in\cL'$ we have that 
	$L'K'$ induces a $k$-walk connecting $L'$ and $K'$ without internal vertices.

	Next we connect $K$ with every $L'\in\cL'$ by a $k$-walk. Again we infer from~\eqref{claim1:N} that we have 
	$|N(V(K))|\ge \frac n{k+1}$ and $|N(V(L'))|\ge \frac n{k+1}$. It, therefore, follows from~\eqref{eq:jcliques}
	that
	\[
		\sum_{u\in N(V(K))}\sum_{w\in N(V(L'))}  \big|\big\{M: M\cong K_k \text{ and } M\subseteq G[N(u,w)]\big\}\big|
		\ge 
		\frac{\xi n^{k+2}}{(k+1)^2}\,.
	\]
	By a double counting argument, this implies that there are at least $\xi'n^k$ $k$-cliques $M$ in $G$
	for which 
	\[
	\big|\big\{ (u,w)\in N(V(K))\times N(V(L')) :   M\subseteq G[N(u,w)]\big\}\big| \ge \xi'n^2.
	\]
For fixed such $M$ we let $U_{K}$ denote the set of those vertices $u\in N(V(K))$ 
	that belong to at least one such pair and let $W_{L'}\subseteq N(V(L'))$
	be defined in the same way. Clearly, we have $|U_{K}|$, $|W_{L'}|\ge \xi'n$.

	We connect $K$ and $L'$ by  $k$-walks through $M$. For the $k$-walk  connecting $K$ and $M$ 
	we shall use the properties of $u\in U_{K}$ and, analogously, we rely on the  properties of 
	$w\in W_{L'}$ for the $k$-walk connecting $M$ and $L'$. 
	Recall that for every $u\in U_{K}$ we have $K\subseteq G[N(u)]$,
	$M\subseteq G[N(u)]$ and an application of~\eqref{claim1:d} with $j=1$ gives
	\[
		\delta(G[N(u)])
		\ge
		\Big(\frac {k-1}{k}+\eps \Big)|N(u)|\,.
	\]
	Thus, by the inductive assumption, there are at least $\rho_{k-1}n^{\ell_{k-1}}$ $(k-1)$-walks  
	connecting the last $k-1$ vertices of $K$
	and the first $k-1$ vertices of $M$ and each such walk has $\ell_{k-1}$ internal vertices. 
	Let $K_+\subseteq K$ be the ordered $(k-1)$-clique spanned by the last $k-1$ vertices of~$K$ and
	let $M_-\subseteq M$ be the ordered $(k-1)$-clique spanned by the first $k-1$ vertices of~$M$.
	Repeating this argument for every  vertex $u\in U_{K}$ 
	we obtain at least 
	\[
		|U_{K}|\cdot\rho_{k-1}n^{\ell_{k-1}}
		\ge 
		\xi'\rho_{k-1}n^{1+\ell_{k-1}}
		=
		\xi''n^{1+\ell_{k-1}}
	\]
	pairs $(u,P)$ where $u\in U_{K}$ and $P$ is a $(k-1)$-walk connecting $K_+$ and $M_-$ 
	in~$G[N(u)]$. As there are no more than $n^{\ell_{k-1}}$ such walks in $G$, 
	there are at least $\tfrac12\xi''n$ vertices $u\in U_{K}$ for which~$G[N(u)]$
	contains at least 
	$\tfrac12\xi''n^{\ell_{k-1}}$ of these walks.
Let us fix one such  $(k-1)$-walk~$P$ and  denote by $U^P_{K}$ the subset of $U_{K}$ consisting of the 
	vertices~$u$ such that~$P\subseteq G[N(u)]$. Next we construct a $k$-walk $Q$ from $K$ to $M$ by 
	inserting 
	\[
		\frac{\ell_{k-1}}{k}+1=2^{k}-1
	\]
	vertices from $U^P_{K}$ into $P$ in such a way that there are 
	exactly~$k$ internal vertices of the $(k-1)$-walk $P$ between 
	each consecutive pair of the vertices of $U^P_{K}$ 
	(see Figure~\ref{fig:connecting}).

	\begin{figure}[ht]
\centering
\begin{tikzpicture}[scale=1.05]
	\coordinate (a1) at (0,0);
	\coordinate (a2) at (2,0);
	\coordinate (y_1) at (4,0);
	\coordinate (y_2) at (6,0);
	\coordinate (y_3) at (8,0);
	\coordinate (y_4) at (10,0);
	\coordinate (b1) at (12,0);
	\coordinate (b2) at (14,0);
	\coordinate (u1) at (3,2);
	\coordinate (u2) at (7,2);
	\coordinate (u3) at (11,2);
	
	\begin{pgfonlayer}{front}
		\foreach \i in {a1, a2, y_1, y_2, y_3, y_4, b1, b2, u1, u2, u3}
			 \fill  (\i) circle (2.2pt);  
		\node[green!90!white, style={font=\vphantom{\Large $K'_2$}}] at (1,-1) {$2$-clique $K$};
		\node[green!90!white, style={font=\vphantom{\Large $K'_2$}}] at (13,-1) {$2$-clique $M$};
		\node[blue!90!white, style={font=\vphantom{\Large $K'_2$}}] at (7,-1) {$1$-path $P$};
		\node[below,black, style={font=\vphantom{\Large $K'_2$}}] at (a2) {$K_+$};
		\node[below,black, style={font=\vphantom{\Large $K'_2$}}] at (b1) {$M_-$};
		\node[red!50!black] at (9,2) {$U_{K}^P$};
	\end{pgfonlayer}
	
	\draw[black!60!white, line width=1.8pt] (a1) -- (b2);
	
		\foreach \i in {a1, a2, y_1, y_2}
			\draw[black!60!white, line width=1.8pt] (u1) -- (\i);
		\foreach \i in { y_1, y_2, y_3, y_4}
			\draw[black!60!white, line width=1.8pt] (u2) -- (\i);	
		\foreach \i in {y_3, y_4, b1, b2}
			\draw[black!60!white, line width=1.8pt] (u3) -- (\i);

	\draw[red!75!black, line width=1.5pt] (u2) ellipse (4.8cm and 15pt);
	\fill[red!75!white,opacity=0.2] (u2) ellipse (4.8cm and 15pt);
	
	\draw [line width=12pt, blue!80!white, opacity=0.3, rounded corners=5pt, line cap=round]
		($(a2)+(180:4pt)$) -- ($(b1)+(0:4pt)$);	
	\draw [line width=12pt, green!80!white, opacity=0.3, rounded corners=5pt, line cap=round]
		($(a1)+(180:0.2)$) -- ($(a2)+(0:0.2)$);
	\draw [line width=12pt, green!80!white, opacity=0.3, rounded corners=5pt, line cap=round]
		($(b1)+(180:0.2)$) -- ($(b2)+(0:0.2)$);
		
	\draw[red!60!yellow, line width=1.2pt, dashed, line cap=round,->] (3,1.77) -- (3,0.25);
	\draw[red!60!yellow, line width=1.2pt, dashed, line cap=round,->] (7,1.77) -- (7,0.25);
	\draw[red!60!yellow, line width=1.2pt, dashed, line cap=round,->] (11,1.77) -- (11,0.25);

\end{tikzpicture}
\caption{Building a $2$-path~$Q$ for $k=1$ that connects 
	$K$ and $M$ by adding vertices from $U_{K}^P$ to a
	$1$-path $P$ connecting $K_+$ and $M_-$ at the indicated places.}
\label{fig:connecting}
\end{figure}
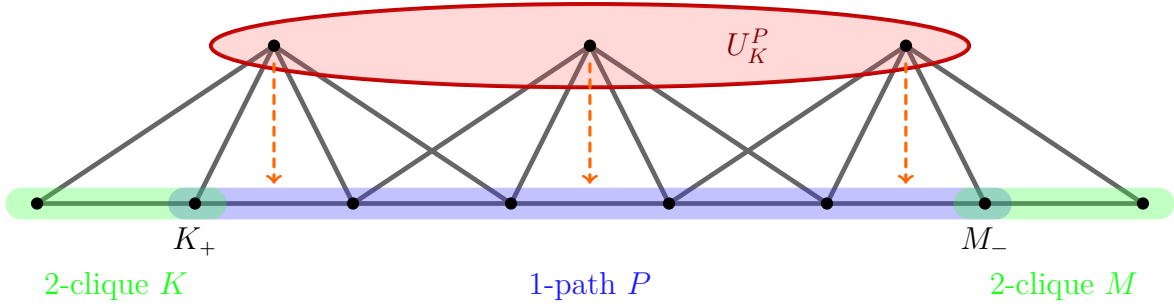

	Note that any such $k$-walk $Q$ created this way is indeed a $k$-walk
	connecting $K$ and $M$ including the first vertex of $K$ and the last vertex of $M$, 
	as every vertex $u\in U^P_{K}\subseteq U_{K}$ contains $K$ and $M$ in its neighbourhood.
	Note that this way we ensure the existence of at least 
	\[
		\frac12\xi''n^{\ell_{k-1}}\cdot \Big(\frac12\xi''n\Big)^{2^{k}-1} 
		= 
		\xi'''n^{\ell_{k-1}+2^{k}-1}
	\] 
	$k$-walks connecting $K$ and $M$.
	
	The same argument applied for $M$ and $L'$ (instead of $K$ and $M$) using the set $W_{L'}$ 
	yields $\xi'''n^{\ell_{k-1}+2^{k}-1}$ 
	$k$-walks connecting $M$ and $L'$. Consequently, for fixed $M$ and $L'$ we obtain 
	\[
		\big(\xi'''n^{\ell_{k-1}+2^{k}-1}\big)^2
	\] 
	$k$-walks connecting $K$ and $K'$ that pass through~$M$ and~$L'$. We recall that there are at least 
	$\xi'n^k$ choices for the clique~$M$ for fixed $L'\in\cL'$ and that $|\cL'|\geq \big(\frac{n}{k+1}\big)^k$
	(see~\eqref{eq:cL'}). Therefore, the number of $k$-walks connecting $K$ and $K'$ is at least
	\[
		\big(\xi'''\big)^2n^{2\ell_{k-1}+2^{k+1}-2}\cdot \xi'n^k\cdot \Big(\frac{n}{k+1}\Big)^k
		\geq 
		\rho_k n^{2\ell_{k-1}+2^{k+1}-2+2k}
		=
		\rho_k n^{\l_k}\,,
	\]
	where the last identity follows from $\ell_{k-1}=k(2^{k}-2)$, which gives indeed
	\[
		2\ell_{k-1}+2^{k+1}-2+2k
		=
		2k(2^k-2+1)+2^{k+1}-2
		=
		(k+1)(2^{k+1}-2)
		=\l_k\,.
	\]
	This concludes the inductive step and the proof of Lemma~\ref{lem:dcon}.
\end{proof}

It is left to deduce Proposition~\ref{prop:connect} from Lemma~\ref{lem:dcon}. Roughly speaking, 
Lemma~\ref{lem:dcon} verifies the assumptions of Corollary~\ref{cor:janson2}, which then 
guarantees that at least one given $k$-path will be enriched to a $(k+1)$-path
by the random graph $G(n,p)$.

\begin{proof}[Proof of Proposition~\ref{prop:connect}]
	Let $k\geq0$ and $\eps>0$ be given. If $k=0$, then we set $\rho_0=1$, 
	and for $k\geq 1$, we appeal to Lemma~\ref{lem:dcon} applied with $k$ and $\eps/2$ and obtain 
	a constant $\rho_k>0$.  We then let $C>1$ be given by Corollary~\ref{cor:janson2}
	applied with 
	\[
		k\,,\qquad
		\l=(k+1)(2^{k+1}-2)\,,\qqand
		\rho=\frac{1}{2^{\l+1}}\rho_k\cdot\Big(\frac{\eps}{2}\Big)^{2k+2}\,.
	\]
	Finally, let $G=(V,E)$ be an $n$-vertex 
	graph with $\delta(G)\geq\big(\frac{k}{k+1}+\eps\big)n$ and $p\geq C/n$.
	
	Consider a set $Z\subseteq V$ of size at most $\eps n/2$ and let $K$ and $K'$ be two disjoint, 
	ordered $(k+1)$-cliques in $G-Z$. 
In order to meet the assumptions of Corollary~\ref{cor:janson2} we first show that there are many ways to greedily extend~$K$ 
	and~$K'$ to $L=x_1\dots x_{k+1}$ and $L'=x'_1\dots x'_{k+1}$ and then we use Lemma~\ref{lem:dcon} to show that there are 
	many possibilities to connect~$L$ and~$L'$.
	
	We remedy this by first selecting 
	$(k+1)$-cliques $L$ and $L'$ in $G-Z$ such that $KL$ and $L'K$ form $(k+1)$-walks. 
	In fact, since $\delta(G-Z)\geq \big(\frac{k}{k+1}+\frac{\eps}{2}\big)n$ we infer that 
	$k+1$ applications of~\eqref{claim1:N} for $j=k+1$ in $G-Z$ give rise to at least 
	$(\eps n/2)^{k+1}$ such ordered $(k+1)$-cliques~$L$. Similarly, there are at least $(\eps n/2)^{k+1}$
	such ordered $(k+1)$-cliques $L'$. For two such ordered cliques $L$ and $L'$ let $L_+$ be 
	the last $k$ vertices in $L$ and let $L'_-$ be the first $k$ vertices in $L'$.

	For $k\geq 1$ the graph $G-Z$ satisfies the assumption of Lemma~\ref{lem:dcon} with~$\eps/2$ 
	instead of $\eps$ and, hence, the lemma yields $\rho_k |V\sm Z|^{\l}$ $k$-walks connecting $L_+$ 
	and~$L'_-$ in $G-Z$ with $\l=(k+1)(2^{k+1}-2)$ internal vertices. 
	For $k=0$ we have $\l=0$ and we note that for the $0$-cliques~$L_+$ and~$L'_-$   
	and the empty path might be considered as a $0$-path connecting those.  
	
	Consequently, for any value of $k$ there are $\rho_k |V\sm Z|^{\l}$ $k$-walks 
	connecting~$L_+$ and~$L'_-$ for all considered  $(k+1)$-cliques $L$ and $L'$. 
	Going over all such $(k+1)$-cliques $L$ and $L'$ this gives rise to 
	\[
		\Big(\frac{\eps}{2} n\Big)^{2k+2}\cdot \rho_k\Big(\frac{1}{2} n\Big)^\l
	\]
	such $k$-walks $x_1\dots x_{k+1}y_1\dots y_{\l}x'_1\dots x'_{k+1}$.
	Since at most $(2k+\l)\l n^{\l-1}$ of these $k$-walks may repeat a vertex, i.e., walks where the vertices $y_1,\dots,y_\l$
	are not pairwise different or  they are not distinct from $K$ or from $K'$,  for sufficiently large $n$, 
	we may assume that at least half of these $k$-walks are indeed $k$-paths disjoint from $K$ and $K'$. This verifies the 
	assumptions of Corollary~\ref{cor:janson2}, which with probability at least 
	$1-4^{-n}$ yields a desired $(k+1)$-path connecting $K$ and $K'$ in $H=G\cup G(n,p)$
	
	Finally, the union bound over up to at most $n^{2k+2}$ choices for $K$ and $K'$ 
	and at most $2^n$ choices for $Z$ shows that a.a.s.\  $H=G\cup G(n,p)$ 
	enjoys the conclusion of Proposition~\ref{prop:connect}.
\end{proof}

\section{Proof of the Reservoir Lemma} \label{sec:reservoir}

\begin{proof}[Proof of Proposition~\ref{prop:reservoir}]
	Consider a random subset $R\subseteq V$ with $|R|=\gamma^2 n$ chosen uniformly at 
	random. Since $\delta(G)\ge(\frac k{k+1}+\eps)n$, it follows from a version of Chernoff's
	inequality appropriate for hypergeometric distributions that for each vertex $v\in V$
	the bad event that $|N(v)\cap R|<(\frac k{k+1}+\frac\eps2)|R|$ holds has probability 
	$\eu^{-\Omega(n)}$. Thus, by the union bound, the probability that there exists some $v\in V$
	for which this bad event occurs is $o(1)$. 
	
	This proves, in particular, that there exists some set $R\subseteq V$ 
	with $|R|=\gamma^2 n$ and 
\begin{equation}\label{eq:Rmindeg}
		 |N(v)\cap R|\ge \left(\frac k{k+1}+\frac\eps2\right)|R|
		 \quad
		 \text{for every $v\in V$.} 
	\end{equation}

	For the rest of the proof we fix some such set $R\subseteq V$ having these properties.
	Notice that~\eqref{eq:Rmindeg} immediately entails that 
\begin{equation}\label{eq:RU}
		|N(J)\cap R|\ge \tfrac12 \eps |R| 
		\quad 
		\text{holds for all $J\in V^{(k+1)}$.}
	\end{equation}

	Let us now fix two ordered $(k+1)$-cliques $K$ and $K'$ in $G-R$ as well as a subset
	$S\subseteq R$ with $|S|\le \frac14\eps|R|$. Consider the bad event $\cE$ that there is no
	$(k+1)$-path in $H$ connecting~$K$ with $K'$ having 
\[
		\l=(k+1)2^{k+1}
	\]
internal vertices 
	all of which belong to $R\sm S$. It suffices to prove that
\begin{equation}\label{eq:PrE}
		\PP(\cE)\leq 4^{-n}\,.
	\end{equation}
This is because there are at most $n^{k+1}$ possibilities for each of $K$ and $K'$ 
	and at most $2^n$ possibilities for $S$, meaning that once~\eqref{eq:PrE} is established
	it will follow that the probability that $H$ fails to have the desired property 
	is at most $n^{2k+2}2^n\cdot o(4^{-n})=o(1)$, as desired. 
	
	For the proof of~\eqref{eq:PrE} we note that due to~\eqref{eq:RU} we can greedily
	extend $K$ to a $(k+1)$-path~$KL$, where $L$ denotes some ordered $(k+1)$-clique 
	in $G[R\sm S]$. More precisely, since $|S|\le \frac 14\eps |R|$ each vertex of
	such a clique $L$ can be chosen in at least $\frac 14\eps |R|$ many ways
	and thus the set $\cL$ containing all such cliques $L$ satisfies $|\cL|\ge (\frac 14\eps |R|)^{k+1}$.
	
	Applying the same reasoning to backwards extensions of $K'$ we infer that the set 
	$\cL'$ consisting of all ordered $(k+1)$-cliques $L'$ in $G[R\sm S]$ for which 
	$L'K'$ is a $(k+1)$-path in~$G$ has at least the size $|\cL'|\ge (\frac 14\eps |R|)^{k+1}$.
	
	Now let $\cP$ be the collection of all $k$-paths in $G[R\sm S]$ having $\l$
	vertices that start with a member of $\cL$ and end with a member of $\cL'$.
	To derive a lower bound on $|\cP|$ we note that as a consequence of~\eqref{eq:Rmindeg} 
	the graph $G[R\sm S]$ satisfies the assumptions of Lemma~\ref{lem:dcon} with~$\eps/4$ here in place of $\eps$ there. Thus for some sufficiently 
	small choice of $\rho_k>0$, Lemma~\ref{lem:dcon} guarantees that for every~$L\in \cL$ 
	and $L'\in\cL'$ there are at least 
	$\rho_k |R\sm S|^{\ell-2k-2}$ $k$-walks with $\ell-2k-2$ internal vertices connecting the last $k$ 
	vertices of $L$ 
	with the first $k$ vertices of~$L'$. 
	Without loss of generality we may assume that $\rho_k\ll\eps^{2k+2}/2^\l$ and 
	since most of these walks are indeed paths for sufficiently large~$n$, this shows that
\begin{equation}\label{eq:cPlarge}
		|\cP|\ge \frac{\rho_k}{2}|\cL||\cL'||R\sm S|^{\ell-2k-2}
		\ge 
		\frac{\rho_k}{2}\left(\frac{\eps}4\right)^{2k+2}\left(1-\frac\eps4\right)^{\ell} |R|^{\l}
		\ge
		\rho_k^2 |R|^{\l}\,.
	\end{equation}
Consequently, we can invoke Corollary~\ref{cor:janson2} for $\l$, $k$, and $\rho=\rho_k^2$, which 
	yields~\eqref{eq:PrE} and thereby Proposition~\ref{prop:reservoir} is proved.
	\end{proof}
	
\section{Proof of the Absorbing Lemma} \label{sec:absorb}
The present section is dedicated to the proof of Proposition~\ref{prop:absorbing}.
As in many earlier applications of the absorbing method the core idea is to take 
a random collection of $\Omega(n)$ small configurations called {\it absorbers}, which 
are then connected by means of the Connecting Lemma to form the desired path $A$. 

The absorbers we shall use later will simply be $(k+1)$-paths on $2k+2$ vertices. 
When such a path $P$ appears in the neighbourhood of some vertex $x$, we have the 
liberty to insert~$x$ in the middle of $P$, thus creating a longer $(k+1)$-path.
In other words, the path~$P$ can {\it absorb} $x$. Now the plan is to construct $A$ 
so as to contain many disjoint absorbers and to make sure that for every $x\in V$ 
there will be at least $2\gamma^2n$ absorbers in $A$ capable of absorbing $x$. 

For standard reasons in the area detailed more fully below, the task of proving 
Proposition~\ref{prop:absorbing} gets thus reduced to estimating the number of 
such absorbers in $H$. This requires to deal with the interplay of the deterministic 
part $G$ and the random part $G(n, p)$ of~$H$. It turns out to be convenient to 
insist that our absorbers are entirely contained in~$G$, except for their ``middle edges'', 
which will have to be taken from $G(n, p)$. It thus becomes necessary to argue that 
$G(n, p)$ is likely to ``complete'' many $x$-absorbers for every $x\in V$ and for 
doing so we exhibit auxiliary graphs $B_x$ with $\Omega(n^2)$ edges and show that 
a.a.s.~$G(n, p)$ intersects each of them in $\Omega(n)$ edges. 

Accordingly, the proof of Proposition~\ref{prop:absorbing} consists of four steps. 
\begin{enumerate}
	\item[$\bullet$] Define for each $x\in V$ a graph $B_x$ on $V$ of size $\Omega(n^2)$
		depending only on $G$.
	\item[$\bullet$] State properties $G(n, p)$ is likely to have that will imply 
		the existence of $A$ in a deterministic sense. 
	\item[$\bullet$] Perform a random selection of $\Omega(n)$ absorbers.
	\item[$\bullet$] Connect these absorbers, thus obtaining $A$.
\end{enumerate}

\begin{proof}[Proof of Proposition~\ref{prop:absorbing}]
	We work with a hierarchy 
\[
		k^{-1}, \eps \gg \beta \gg \gamma \gg C^{-1}
	\]
and we consider an $n$-vertex graph $G=(V,E)$ with $\delta(G)\geq (\frac{k}{k+1}+\eps)n$.
	
	\paragraph{The graphs \texorpdfstring{$B_x$}{\it B(x)}}
	Let $P$ denote the $(k+1)$-path on $2k+2$ vertices $1, \ldots, {2k+2}$
	and let~$P^-$ be the graph obtained from $P$ by deleting the middle edge between ${k+1}$ and ${k+2}$.
	Notice that the chromatic number of $P^-$ is (at most) $k+1$, an admissible colouring
	being the map $\phi\colon V(P^-)\lto [k+1]$ assigning the colours 
	$1, \ldots, k+1, k+1, 1, \ldots, k$ in this order to the vertices of $P^-$, 
	i.e., explicitly  
\[
		\phi(i)=
			\begin{cases}
				i     & \text{if $1\le i\le k+1$,} \cr
				k+1   & \text{if $i=k+2$,} \cr
				i-k-2 & \text{if $k+3\le i\le 2k+2$.}
			\end{cases}
	\]

	We claim that for every vertex $x\in V$ there are at least 
\begin{equation}\label{eq:NsmR}
		\beta n^{2k+2} \text{ ordered copies of $P^-$ in } G[N(x)]\,.
	\end{equation}
This is clear for $k=0$, as in this case the graph $P^-$ has two vertices and no edges. 
	If~$k>0$ we apply Lemma~\ref{lem:scale} to $J=\{x\}$ and learn that the graph $G[N(x)]$ 
	has order at least $\frac{k}{k+1}n$ and minimum degree at least
	$(\frac{k-1}{k}+\eps)|N(x)|$. So by the Erd\H os--Stone theorem there is at least one copy of $P^-$ in 
	$G[N(x)]$ and by supersaturation 
	(see, e.g.,~\cite{ES1983}) there are indeed at least $\beta n^{2k+2}$ ordered copies 
	of $P^-$ in $G[N(x)]$, which completes the proof of~\eqref{eq:NsmR}.  
	
	Let $B_x$ be a graph on $V$ whose edges are 
	the pairs $vv'$ with the property that there are at least $\beta n^{2k}$ injective graph homomorphism 
	$\phi\colon P^-\to G[N(x)]$ with $\phi(k+1)=v$ and $\phi(k+2)=v'$.
	It follows from the discussion above that
\begin{equation}\label{eq:Bxlarge}
		e(B_x)\ge \beta n^2/2\,.
	\end{equation}

	\paragraph{Properties of \texorpdfstring{$G(n, p)$}{\it G(n, p)}} 
	We will now check that the following statements hold a.a.s.
	\begin{enumerate}[label=\rmlabel]
		\item\label{it:a1} $G(n, p)$ has at most $Cn$ edges.
		\item\label{it:a2} There are at most $2C^2n$ ordered pairs $(e, e')$ of intersecting 
			edges in $G(n, p)$.
		\item\label{it:a3} For every $R\subseteq V$ with 
			$|R|\le \gamma^2n$ and every $v\in V$ at least $\beta C n/4$ edges of  
			$B_x-R$ appear in $G(n, p)$.
	\end{enumerate}
	
	Notice that~\ref{it:a1} is straightforward by Chernoff's inequality. 
	For~\ref{it:a2} we remark that the random variable counting such pairs has expected value
	and variance~$O_C(n)$ and, therefore, Chebyshev's inequality applies. 
	Finally, for every $R$ and $x$ as in~\ref{it:a3} we have 
	$e(B_x-R)\ge \beta n^2/2-|R|n\ge \beta n^2/3$ by~\eqref{eq:Bxlarge} and $\gamma\ll \beta$.
	Thus the expected value of the number $X_{R, x}$ of edges that $G(n, p)$ and $B_x-R$ 
	have in common is at least $\beta Cn/3$. In view of Chernoff's inequality 
	(see~\cite{JLR}*{Section 2.1}) and $C\gg \beta^{-1}$
	it follows that 
\[
		\PP(X_{R, x}<\beta Cn/4)<e^{-\beta Cn/96}<4^{-n}\,.
	\]
Taking the union bound over all choices for the pair $(R, x)$ we infer that~\ref{it:a3} 
	fails with a probability of at most $n2^n\cdot 4^{-n}=o(1)$.  
	
	Having thus proved~\ref{it:a1},~\ref{it:a2}, and~\ref{it:a3} to hold a.a.s.\
	we shall henceforth regard $G(n, p)$ as a fixed graph having these properties, 
	for which, moreover, the conclusion of Proposition~\ref{prop:connect} is valid. 
	
	As we shall see, these assumptions imply 
	the existence of the desired absorbing path. Let us fix a set $R\subseteq V$ with 
	$|R|\le \gamma^2n$ from now on.
	
	\paragraph{Selection of absorbers}
	An ordered copy 
	$\seq{v}=(v_1, \ldots, v_{2k+2})\in (V\sm R)^{2k+2}$ of $P^{-}$ in $G-R$  
	with $v_{k+1}v_{k+2}\in E(G(n, p))$ is called an {\it absorber}.  
	Notice that by~\ref{it:a1} there exist at most $Cn^{2k+1}$ absorbers. 
	
	In case all vertices of an absorber $\seq{v}$ are in $N_G(x)$ for some 
	vertex $x\in V$ we say that~$\seq{v}$ is an {\it $x$-absorber}. As explained earlier, 
	the rationale behind this terminology is that if the path~$A$ we are about to 
	construct happens to contain an $x$-absorber $\seq{v}=(v_1, \ldots, v_{2k+2})$, 
	then we may replace this part of $A$ by the $(k+1)$-path 
	$(v_1, \ldots, v_{k+1}, x, v_{k+2}, \ldots, v_{2k+2})$
	whenever we wish to ``absorb'' $x$ into $A$. Later we shall refer to this option as 
	the {\it absorbing property} of~$\seq{v}$. We contend that 
\begin{equation}\label{eq:xabs}
		\text{ for every $x\in V$ there are at least 
				$\beta^2C n^{2k+1}/4$ many $x$-absorbers.}
	\end{equation}

	Notice that by~\ref{it:a3} this would follow from the fact that 
	for every edge $vv'$ that $B_x-R$ and $G(n, p)$ have in common 
	there are at least $\beta n^{2k}/2$ many $x$-absorbers having $v$ and~$v'$
	in their $(k+1)$-st and $(k+2)$-nd position, respectively. Now for 
	$vv'\in e(B_x)$ there are actually at least $\beta n^{2k}$ such configurations
	in $V$ and at most $2k|R|n^{2k-1}$ of them can fail to be $x$-absorbers
	for the reason of containing a vertex from $R$. Due to $|R|\le \gamma^2 n$ 
	and $\gamma\ll\beta$ at most $\beta n^{2k}/2$ candidates get discarded in this way,
	and thereby~\eqref{eq:xabs} is proved. 
	 
	Now let $\cF$ be a random set of absorbers containing each absorber independently
	and uniformly at random with probability $q=\gamma^{3/2}C^{-1}n^{-2k}$. 
	Since 
\[
		\EE[|\cF|]\le Cn^{2k+1}\cdot q = \gamma^{3/2} n\,, 
	\]
Markov's inequality entails
\begin{equation} \label{eq:FF1}
		\PP(|\cF|\le 3\gamma^{3/2} n) > 2/3\,.
	\end{equation}

	An ordered pair $(\seq{v}, \seq{w})$ of absorbers is said to be {\it overlapping}
	if they have a vertex in common. When two absorbers overlap, then either their 
	middle edges are disjoint or they are not. The first case appears at most 
	$(Cn)^2\cdot 4k^2n^{4k-1}$ many times by~\ref{it:a1}, while the second case appears 
	at most $8C^2n\cdot n^{4k}$ times by~\ref{it:a2}. So altogether there are at most 
	$(4k^2+8)C^2n^{4k+1}$ pairs of overlapping absorbers. Hence, the expected number of 
	overlapping pairs $(\seq{v}, \seq{w})\in \cF^2$ is at most $(4k^2+8)\gamma^{3}n$,
	and a further application of Markov's inequality yields 
\begin{equation} \label{eq:FF2}
		\PP(\text{there are at most $\gamma^{5/2}n$ overlapping pairs in $\cF^2$}) > 2/3\,.
	\end{equation}

	Since for each $x\in V$ the expected number of $x$-absorbers in $\cF$
	is by~\eqref{eq:xabs} at least $\beta^2\gamma^{3/2}n/4$, Chernoff's inequality implies	
\begin{equation} \label{eq:FF3}
		\PP(\text{there are at least $3\gamma^2n$ 
			many $x$-absorbers in $\cF$ for every $x\in V$}) > 2/3\,.
	\end{equation}

	In view of~\eqref{eq:FF1},~\eqref{eq:FF2}, and~\eqref{eq:FF3}
	there is an instance $\cF_\star$ of $\cF$ having the three properties 
	whose probabilities were just shown to be larger than $2/3$. Delete from $\cF_\star$
	all absorbers belonging to an overlapping pair and denote the resulting set of 
	absorbers by~$\cF_{\star\star}$. Notice that~$\cF_{\star\star}$ enjoys the following properties 
	\begin{enumerate}
		\item[$\bullet$] $|\cF_{\star\star}|\le 3\gamma^{3/2} n$,
		\item[$\bullet$] no two absorbers in $\cF_{\star\star}$ overlap, and
		\item[$\bullet$] for each $x\in V$ there are at least $2\gamma^2 n$
			many $x$-absorbers in $\cF_{\star\star}$.
	\end{enumerate}
		
	\paragraph{Building the absorbing path} 
	An iterative application of Proposition~\ref{prop:connect} allows us to connect 
	the members of $\cF_{\star\star}$ into a single path $A\subseteq G-R$ with
\[
		|V(A)|
		\le 
		(2k+2)|\cF_{\star\star}|+(k+1)2^{k+1}(|\cF_{\star\star}|-1)
		\le 
		\gamma n/2\,.
	\]

	In each of those applications of the Connecting Lemma, we take $K$ and $K'$ 
	to be end-sets of the two $(k+1)$-paths we wish to connect, and we let $Z$ be 
	the union of the other vertices in the path system we currently have with $R$. 
	Since at every moment the $(k+1)$-paths we are currently dealing with will  
	have at most $\gamma n/2$ vertices in total and $|R|\le \gamma^2n$,
	we will have $|Z|\le \gamma n$ in each of our $|\cF_{\star\star}|-1$
	applications of Proposition~\ref{prop:connect}, as required. 
	
	Using the absorbing property of $x$-absorbers in a greedy manner one 
	sees immediately that the $(k+1)$-path $A$ just constructed has the required property.  
\end{proof}

\section{Proof of the Covering Lemma} \label{sec:cover}

This section deals with the proof of Proposition~\ref{prop:covering}. Roughly speaking,
our strategy is as follows. By known results~\cites{EKT1987, CKO2007} the minimum degree
condition imposed on $G$ is more than enough to guarantee that we can cover essentially 
all vertices of $G'=G-Q$ with vertex-disjoint copies of the graph $K_{k+2}^-$ which arises from 
a clique of order $k+2$ by the deletion of a single edge. A standard application of the
regularity method for graphs would allow to strengthen this result so as to obtain, for any 
bounded number $m$, a covering of an overwhelming proportion of the vertices of $G'$ by 
vertex-disjoint copies of the $m$-blow-up~$K_{k+2}^-(m)$ of a $K_{k+2}^-$. Explicitly, this 
is the graph arising from a $K_{k+2}^-$ upon replacing each of its vertices $x$ by an independent 
set $V_x$ of size $m$ and each of its edges $xy$ by a complete bipartite graph 
$K_{m,m}$ joining $V_x$ and $V_y$. An important point here is that there is a tremendous amount
of flexibility in the construction of such an almost-covering of~$G'$ by copies 
of~$K_{k+2}^-(m)$. 

Now for any $K_{k+2}^-(m)$ in $G$ it may happen that an appropriate 
path on $2m$ vertices in~$G(n, p)$ augments it to a graph containing a spanning $(k+1)$-path
in $H$. Of course, for any particular $K_{k+2}^-(m)$ in $G$ this is an extremely 
unlikely event having a probability of only $o(1)$. However, owing to the aforementioned 
flexibility in the construction of an almost $K_{k+2}^-(m)$-covering of $G'$, it becomes 
asymptotically almost surely possible to ensure that we only take copies $K_{k+2}^-(m)$ 
for which such a path in $G(n, p)$ is available.   

In the two subsequent subsections we provide some of the background alluded to in the 
two foregoing paragraphs, while the proof of Proposition~\ref{prop:covering} will be given 
in Section~\ref{subsec:cov}.

\subsection{\texorpdfstring{$K_r^-$-factors}{Factors}}
For $r\ge 3$ let $K_r^-$ denote the graph obtained from the clique $K_r$ by deleting one edge. 
A $K_r^{-}$-factor of a graph $G$ is a spanning subgraph of $G$ each of whose connected 
components is isomorphic to $K_r^{-}$. It was proved by 
Enomoto, Kaneko, and Tuza~\cite{EKT1987} that every sufficiently large connected graph $G$
with $\delta(G)\ge \frac13|V(G)|$ whose number of vertices is divisible by $3$ contains a
$K_3^-$-factor.  For larger values of $r$ the tight minimum degree condition ensuring 
the existence of a $K_r^-$-factor was determined by Cooley, K\"uhn, and Osthus~\cite{CKO2007}.
By combining the results in those two references one obtains the following. 

\begin{theorem}\label{cko} 
	For every integer $r\ge 3$ there exists an integer $n_0$ such that every connected 
	graph $G$ with $n\ge n_0$ vertices, $r\mid n$, and 
\[
		\delta(G)\ge \left(1-\frac{r-1}{r(r-2)}\right)n
	\]
contains a $K_r^-$-factor.\qed
\end{theorem}
 
For the application we have in mind the following `imperfect' consequence of this result, 
where we omit the divisibility assumption on $n$ and allow a bounded number of left-over 
vertices, will be more convenient. 

\begin{corollary}\label{cor:cool} 
	For every integer $k\ge	1$ there exists $n_0\in\NN$ such that every graph $G$ 
	with~$n\ge n_0$ vertices and
\[
		\delta(G)\ge\left(1-\frac{k+1}{k(k+2)}\right)n
	\]
contains a collection of vertex disjoint copies of $K_{k+2}^-$ which together cover 
	all but at most~$(k+2)^2$ vertices of $G$.
\end{corollary}

\begin{proof}
	We check that the number $n_0$ provided by Theorem~\ref{cko} suffices. 
    Let $r$ be the integer satisfying $0\le r\le k+1$ and $n\equiv r\pmod{k+2}$. 
    Add $k+2-r>0$ new vertices to $G$ and connect them to all other vertices (and 
    to each other). The graph thus obtained satisfies the assumptions of Theorem \ref{cko}, 
    and hence it contains a $K_{k+2}^-$-factor. When returning to $G$ we can `lose' at 
    most $k+2-r$ copies of $K_{k+2}^-$, wherefore at most~$(k+2)^2$ vertices remain 
    uncovered by the `surviving' copies of $K_{k+2}^-$.  
\end{proof}

\subsection{The graph regularity method}
Mainly in order to fix some notation we shall now state a version of Szemer\'edi's 
Regularity Lemma from \cite{Sz1978}.
For two real numbers~$\delta>0$ and $d\in [0, 1]$, a graph~$G$ and two nonempty 
disjoint sets $A, B\subseteq V(G)$, we say
that the pair $(A, B)$ is {\it $(\delta, d)$-quasirandom} if for all $X\subseteq A$
and $Y\subseteq B$ the inequality
\[
	\big|e(X, Y)-d|X||Y|\big|\le \delta |A||B|
\]
holds. The pair $(A, B)$ is {\it $\delta$-quasirandom} if it
is $(\delta, d)$-quasirandom for $d=e(A, B)/|A||B|$.

\begin{theorem}[Szemer\'edi's Regularity Lemma]\label{thm:szem}\samepage
	Given $\delta>0$ and $t_0\in\NN$ there exists an integer $T_0$ such that every graph $G=(V,E)$
	on $n\ge t_0$ vertices admits a partition
		\begin{equation*}V=V_0\dcup V_1\dcup \ldots\dcup V_t
	\end{equation*}
		of its vertex set such that
	\begin{enumerate}[label=\rmlabel]
		\item $t\in [t_0, T_0]$, $|V_0|\le \delta |V|$, and $|V_1|=\ldots=|V_t|$, and 
		\item for every $i\in [t]$ the set
			$\bigl\{j\in [t]\setminus \{i\}\colon (V_i, V_j) \text{ is 
						not $\delta$-quasirandom}\bigr\}$
			has size at most $\delta t$. \qed
	\end{enumerate}
\end{theorem}

Any partition as in Theorem \ref{thm:szem} is called \emph{$\delta$-quasirandom} or 
just \emph{quasirandom}.
In the literature one often finds other versions of the Regularity Lemma, where instead of the
second condition above one requires that at most $\delta t^2$ pairs $(V_i, V_j)$
with distinct $i, j\in [t]$ fail to be $\delta$-quasirandom. Applying such a regularity
lemma to appropriate constants $\delta'\ll \delta$ and $t_0'\gg \max(t_0, \delta^{-1})$, 
and relocating partition
classes involved in many irregular pairs to~$V_0$, one can obtain the version stated here.

Next we state the Counting Lemma accompanying Szemer\'edi's Regularity Lemma.

\begin{lemma}[Counting Lemma] \label{lem:count}
	Let $F$ be a graph with vertex set $[f]$ and let $G$ be another graph with a 
	partition $V(G)=V_1\dcup\dots\dcup V_f$ such that $(V_i,V_j)$ is $\delta$-quasirandom 
	whenever $ij\in F$. Then the number of ordered copies of $F$ in $G$, that is, the number 
	of $f$-tuples $(v_1,\dots,v_f)\in V_1\times\cdots\times V_f$ such that $v_iv_j\in G$ 
	whenever $ij\in F$, equals
\[
		\left(\prod_{ij\in F}d_{ij}\pm e(F)\delta\right)\prod_{i=1}^f|V_i|,
	\]
where $d_{ij}=\frac{e(V_i,V_j)}{|V_i||V_j|}$ is the density of $(V_i,V_j)$, which is set to $0$ in case $V_i=\emptyset$ or 
	$V_j=\emptyset$.\qed
\end{lemma}

\subsection{The covering lemma} \label{subsec:cov}
 
We are now ready for the proof of the covering lemma. 

\begin{proof}[Proof of Proposition~\ref{prop:covering}]
	We begin by choosing several constants fitting into the hierarchy
\[
		k^{-1},\eps \gg \gamma \gg m^{-1}, \delta, t_0^{-1} \gg T_0^{-1} \gg \tau \gg C^{-1}
	\]
	and we consider an $n$-vertex graph $G=(V,E)$ with $\delta(G)\geq (\frac{k}{k+1}+\eps)n$.
	
	Next we describe a deterministic property the random graph $G(n, p)$ for $p=C/n$ is likely 
	to have and the remainder will then be dedicated to showing that this property 
	implies the conclusion of our Covering Lemma in a deterministic way. 
	
	For every sequence $\seq{X}=(X_1,\dots, X_{k+2})$ of disjoint subsets of $V$ 
we define a family ${\mathcal F}({\seq{X}})$ 
	of $2m$-vertex paths with vertex set $V$ as follows. 
	Consider the set of all pairs $(Y_1,Y_2)$ of $m$-sets with $Y_1\subseteq X_1$ and 
	$Y_2\subseteq X_2$ such that there are further $m$-sets $Y_i\subseteq X_i$ 
	for $i\in [3, k+2]$ such that $Y_1\dcup\ldots\dcup Y_{k+2}$ spans a copy 
	of~$K^-_{k+2}(m)$ in $G$ having all $Y_i$-$Y_{i'}$ edges for all $1\le i<i'\le k+2$ 
	with $(i, i')\ne (1, 2)$. For each such pair $(Y_1,Y_2)$ choose a spanning path 
	$P(Y_1,Y_2)$ on $Y_1\dcup Y_2$ that alternates between the two classes. 
	The family  ${\cF}({\seq{X}})$ consists of all 
	these paths taken over all choices of $(Y_1,Y_2)$ as above. 
	Finally, let 
\[
		\ccJ = \bigl\{ \seq{X}=(X_1,\dots, X_{k+2})
		\colon 
		|\cF(\seq{X})|\ge \tau n^{2m}\bigr\}\,. 
	\]

	By Janson's inequality (see Theorem~\ref{thm:janson}), a sufficiently large choice of $C$ guarantees 
\[
		\PP\bigl(\text{$P\not\subseteq G(n, p)$ for all  $P\in \cF(\seq{X})$}\bigr)
		\le
		o(2^{-(k+2)n})
	\]
for each $\seq{X}\in \ccJ$. Since $|\ccJ|\le 2^{(k+2)n}$ holds trivially, the union bound 
	informs us that the event~$\cE$ that for every $\seq{X}\in \ccJ$ there is a path 
	$P\in \cF(\seq{X})$ with $P\subseteq G(n, p)$ has probability~${1-o(1)}$.
	Henceforth we assume that $\cE$ occurs. 
	
	Applying Theorem~\ref{thm:szem}
	to $G'=G-Q$ we obtain for some $t\in [t_0, T_0]$ a $\delta$-quasirandom partition 
\[
		V\sm Q
		=
		V_0 \dcup V_1\dcup\ldots\dcup V_t
	\]
of $G'$. Let $\Gamma$ be the {\it reduced graph} with vertex set $[t]$
	defined in such a way that a pair $ij\in [t]^{(2)}$ forms an edge of $\Gamma$
	if and only if the pair $(V_i, V_j)$ is $\delta$-quasirandom  
	with density $d_{ij}=e(V_i,V_j)/|V_i||V_j|\ge \frac1{(k+1)^2}$. 
	We contend that  
\begin{equation}\label{eq:Gamma}
		\text{if } k>0, \text{ then } 
		\delta(\Gamma)\ge \left(1-\frac{k+1}{k(k+2)}\right)t\,.
	\end{equation}
For the proof of this estimate we consider an arbitrary $i\in [t]$
	and note that the minimum degree condition imposed on $G$ yields 
\[
		e(V_i, V)\ge \left(\frac{k}{k+1}+\eps\right)|V_i|n\,.
	\] 
On the other hand, it readily follows from the definitions of a $\delta$-quasirandom
	partition and~$\Gamma$ that 
\begin{align*}
		e(V_i, V) &\le e(V_i, Q\cup V_0)+\delta t |V_i|^2+ d_\Gamma(i) |V_i|^2
			 	+(t-d_\Gamma(i))\frac{|V_i|^2}{(k+1)^2} \\
		&\le (\gamma+\delta)|V_i|n+\delta |V_i|n + \frac{1}{(k+1)^2}|V_i|n
			+\frac{k(k+2)}{(k+1)^2}\cdot \frac{d_\Gamma(i)}{t}\cdot |V_i|n
	\end{align*}
Provided that $\gamma+2\delta\le \eps$ the combination of both estimates 
	yields
\[
		\frac{d_\Gamma(i)}{t}
		\ge 
		\frac{(k+1)^2}{k(k+2)}\left(\frac{k}{k+1}-\frac{1}{(k+1)^2}\right)
		=
		1-\frac{k+1}{k(k+2)}
	\]
and thereby~\eqref{eq:Gamma} is proved.
	
	Now the main work that remains to be done is to show the following statement.
	
	\begin{claim}\label{clm:K}
		If $K\subseteq V(\Gamma)$ induces a $K_{k+2}^-$ and $V_K=\bigcup_{i\in K}V_i$,
		then all but at most $\tfrac12 \gamma^2|V_K|$ vertices of $H[V_K]$ can be 
		covered by a family of vertex disjoint $(k+1)$-paths each on $(k+2)m$ vertices.
	\end{claim}
	
	Assuming for the moment that we already know this, the proof of 
	Proposition~\ref{prop:covering} can be completed as follows. 
	If $k\ge 1$, then by Corollary~\ref{cor:cool} and~\eqref{eq:Gamma} we know that~$\Gamma$
	contains an almost perfect $K_{k+2}^-$-factor $\cK$ covering all but at most
	$(k+2)^2$ vertices of $\Gamma$. As a $K_2^-$ is the empty graph on two vertices, 
	such a factor $\cK$ exists for $k=0$ as well. Applying Claim~\ref{clm:K} to each 
	$K_{k+2}^-$ in $\cK$ we obtain a family of vertex disjoint $(k+1)$-paths in 
	$H-Q$ covering all but at most $\bigl(\delta+\frac{(k+2)^2}{t_0}+\tfrac12 \gamma^2\bigr)n$
	vertices, and by $\delta, t_0^{-1}\ll \gamma$ this is at most $\gamma^2n$. Moreover,
	the number of these $(k+1)$-paths can be at most $\frac{n}{(k+2)m}$, which by 
	$\gamma\gg m^{-1}$ is indeed at most $\gamma^3n$.	
	
	It remains to prove Claim~\ref{clm:K}. To this end we may suppose that $V(K)=[k+2]$ 
	and that the (perhaps)
	missing edge of the $K_{k+2}^-$ is $\{1, 2\}$. Let $\cP$ be a maximum collection of 
	vertex-disjoint $(k+1)$-paths with $(k+2)m$ vertices in the $(k+2)$-partite graph 
	$H[V_1, \ldots, V_{k+2}]$.
	For each $i\in [k+2]$ let $X_i\subseteq V_i$ be the set of vertices in $V_i$ which are 
	not used by these paths. Since each path in $\cP$ needs to consist of $m$ vertices from 
	each~$V_i$, it follows that $|X_1|=\ldots=|X_{k+2}|=x$ holds for some integer $x$. 
	Now it suffices to prove $x\le \frac 12 \gamma^2|V_1|$, so assume for the sake of 
	contradiction that this fails. 

	We intend to derive $|\cF(\seq{X})|\ge \tau n^{2m}$ from the alleged largeness of $x$, 
	which will tell us that $\seq{X}\in \ccJ$.
	To this end we shall first obtain a lower bound on the number $\Omega$ of copies 
	of~$K^-_{k+2}(m)$ in $G[X_1, \ldots, X_{k+2}]$ having 
	\begin{enumerate}
		\item[$\bullet$] $m$ vertices in each $X_i$ and 
		\item[$\bullet$] all edges between the vertices in $X_i$ and $X_{i'}$ 
			for $1\le i<i'\le k+2$ with $(i, i')\ne (1, 2)$. 
	\end{enumerate}
	For each $i\in [k+2]$ let $X_i=X_{i,1}\dcup\ldots\dcup X_{i, m}$ be a partition 
	of $X_i$ into $m$ sets of size $x/m$. Now for $1\le i<i'\le k+2$ with $(i, i')\ne (1, 2)$
	we have $ii'\in E(\Gamma)$, which indicates that the pair $(V_i, V_{i'})$ is 
	$(\delta, d_{ii'})$-quasirandom in $G$ for some $d_{ii'}\in[\frac 1{(k+2)^2}, 1]$. 
	For $j, j'\in [m]$ we have $|X_{ij}|\ge \frac{\gamma^2}{2m}|V_i|$ and
	$|X_{i'j'}|\ge \frac{\gamma^2}{2m}|V_j|$ by our indirect assumption on $x=|X_i|=|X_j|$
	and thus the pair $(X_{ij}, X_{i'j'})$ is $(\delta^\star, d_{ii'})$-quasirandom
	in $G$, where $\delta^\star=\frac{4m^2\delta}{\gamma^4}$. By Lemma~\ref{lem:count} 
	applied to $F=K_{k+2}(m)$ and the vertex classes $X_{ij}$ with $i\in[k+2]$ and $j\in [m]$
	it follows that 
\[
		\Omega
		\ge 
		\left(\frac 1{(k+2)^{2e(K_{k+2}^-(m))}}-e\bigl(K_{k+2}^-(m)\bigr)\delta^\star\right)
		\left(\frac{x}{m}\right)^{m(k+2)}\,,
	\]
which by $\tau\ll \delta, m^{-1}, T_0^{-1}\ll \gamma\ll k^{-1}$ gives $\Omega\ge \tau n^{m(k+2)}$. 
	In particular, there are at least~$\tau n^{2m}$ pairs of $m$-sets $(Y_1, Y_2)$
	with $Y_1\subseteq X_1$ and $Y_2\subseteq X_2$ which can be completed to a copy 
	of~$K_{k+2}^-(m)$ in $G[X_1, \ldots, X_{k+2}]$ by appropriate further $m$-sets 
	$Y_i\subseteq X_i$ for $i\in [3, k+2]$. For these reasons, we have 
	indeed $|\cF(\seq{X})|\ge \tau n^{2m}$ and $\seq{X}\in \ccJ$.
	
	Thus the occurrence of $\cE$ supplies a path $P\in \cF(\seq{X})$ with 
	$P\subseteq G(n, p)$. For $i\in [3, k+2]$ let $Y_i\subseteq X_i$ be $m$-sets 
	witnessing $P\in \cF(\seq{X})$. Since ${V(P)\cup Y_3\cup\ldots\cup Y_{k+2}}$
	spans a $(k+1)$-path in $H[X_1, \ldots, X_{k+2}]$, we get a contradiction to the 
	maximality of the collection $\cP$ chosen earlier. 
	This concludes the proof of Claim~\ref{clm:K}
	and, hence, the proof of Proposition~\ref{prop:covering}.  
\end{proof}

\subsection*{Note added in proof}
After the present work was submitted, Nenadov and Truji\'c~\cite{NT2018} showed that Theorem~\ref{thm:main} is true even if one replaces 
$\cP^{k+1}$ by $\cP^{2k+1}$. Independently, a more general result was recently obtained by Antoniuk and the current authors~\cite{ADRRS2019}.

\end{document}